\documentclass[11pt]{article}
\usepackage[utf8]{inputenc}

\usepackage[a4paper,top=3cm,bottom=2cm,left=3cm,right=3cm,marginparwidth=1.75cm]{geometry}

\usepackage[toc,page]{appendix}
\usepackage{graphicx} 
\usepackage{amsthm}
\usepackage{amsmath}
\usepackage{amssymb}
\usepackage{tikz}
\usepackage{authblk}
\usepackage{subcaption}
\usepackage{float} 
\usetikzlibrary{calc, chains,
                fit,
                positioning,
                shapes,
                decorations.pathreplacing}

\newtheorem{theorem}{Theorem}[section]
\newtheorem{corollary}{Corollary}[section]
\newtheorem{lemma}{Lemma}[section]
\newtheorem{proposition}{Proposition}[section]
\newtheorem{remark}{Remark}[section]
\newtheorem{note}{Note}[section]

\newtheorem{definition}{Definition}[section]

\newcommand{\mr}{\mathrm}

\newcommand{\DN}{\mathrm{DN}}

\newcommand{\s}{\mathsf}

\newcommand{\D}{\mathcal{D}}

\author{Shivjyot Brar, Sheng-Chang Chen, Sayonita Ghosh Hajra, Santosh Kandel}

\title{Gluing Formula for the Pseudo-Determinant of Graph Laplacian and Applications to Counting of Spanning Trees}

\begin{document}

\date{}
\maketitle

\begin{abstract}
\noindent
    In this paper, we establish a gluing formula for the pseudo-determinant of the Laplacian on a simple finite graph. We achieve this by using the gluing formula for the determinant of massive Laplacian and the perturbation theory technique. In addition, we apply this gluing relation to derive a gluing formula for the number of spanning trees and rooted spanning forests on simple finite graphs.
\end{abstract}



\section{Introduction}

\noindent

The pseudo-determinant of a square matrix is the product of its non-zero eigenvalues. If the square matrix is invertible, then its pseudo-determinant is its determinant. For a detailed discussion of pseudo-determinants, we refer to \cite{Knill2014}, where, among many results, a Cauchy–Binet formula for pseudo-determinants is established.

In combinatorial graph theory, the pseudo-determinant plays a prominent role, as the ratio of the pseudo-determinant of the graph Laplacian to the number of vertices counts the number of spanning trees \cite{Kelmans1974, kirchhoff1847, Klee2019}.
Counting the number of spanning trees is a central problem in combinatorial graph theory and has been extensively studied in the literature \cite{gross2015, Zhou2020}. The number of spanning trees of the complete graph $K_n$ is given by the Cayley's formula $n^{n-2}$ \cite{Cayley2009}. A generalization of Cayley's formula for regular graphs is considered in \cite{sachs1962}. Another generalization is the so-called Weinberg formula \cite{weinberg1958, moon1970}, which computes the number of spanning trees of a graph obtained from the complete graph by removing the edges incident to a single vertex. Many techniques have been developed to compute explicit formulas for the number of spanning trees in various classes of graphs \cite{Asaner2022, Klee2019, Zhou2020}.

Finding an explicit formula for the number of spanning trees of a graph is not feasible in general. If a graph has a small number of vertices, the number of spanning trees can be computed directly. However, when a graph has a large number of vertices, computing the number of spanning trees becomes increasingly tedious. One approach to systematically address this issue is to decompose the graph into smaller subgraphs and then express the number of spanning trees in terms of those associated with the smaller pieces.

 The idea of ``cutting'' and ``gluing'' is similar to the Mayer-Vietoris theorem and Seifert–Van Kampen theorem in Topology \cite{Hatcher2002}. In Quantum Field Theory (QFT) literature, it is known as the locality principle \cite{reshetikhin2015}. In QFT, one of the key quantities of interest is the so-called partition function on a space-time, which, in simplified cases, is given by a Gaussian integral. The locality principle suggests that the partition function on a space–time can be obtained by gluing together the partition functions on smaller pieces obtained by cutting the space–time. Developing a rigorous mathematical framework for QFT based on the locality principle is, in fact, a very active area of research. We refer interested readers to \cite{reshetikhin2015, contreras2024} and references therein.

From the QFT perspective, determinants can be related to partition functions of a Gaussian theory on a space–time \cite{reshetikhin2015, contreras2024}.  Computing the number of spanning trees and spanning forests involves evaluating determinants, and hence such quantities can be represented using partition functions. Connections between combinatorial quantities and partition functions in statistical mechanics have been explored in the literature \cite{DD1988, Kenyon2011}.  Reinterpretations of combinatorial quantities as Grassmann integrals and generalizations of Kirchhoff’s Matrix–Tree theorem have also been studied \cite{caracciolo2004}.  While many studies relate combinatorial objects to quantities in statistical mechanics and QFT,  to the best of our knowledge,  the application of the locality principle to combinatorial quantities has not received much explicit attention.

The primary objective of this work is to establish a gluing formula for the pseudo-determinant of graph Laplacian and then use it to deduce a gluing formula for the spanning tree enumerator.  We also establish Schur complement type formula for the pseudo-determinant of graph Laplacian and show that it is equivalent to Schur complement type formula for the spanning tree enumerator derived in \cite{Zhou2020}. Thus,  this work complements the work in \cite{Zhou2020}. Our approach uses a gluing formula for the determinant of the massive Laplacian \cite{reshetikhin2015, contreras2024} followed by a perturbation-theoretic argument. In addition,  we study the Dirichlet-to-Neumann map associated to the characteristics matrix of the Laplacian on a finite simple graph, which can be of independent interest.

The secondary objective of this work is to provide more conceptual insight on the gluing formula for the spanning tree enumerator using the locality principle in QFT.  It is possible to relate determinant of the determinant of the massive Laplacian to the partition function of a Gaussian QFT on the finite graph.  Here,  the role of  space–time  is played by a finite graph.  Then an application of locality principle yields the gluing formula for the determinant of the massive Laplacian \cite{reshetikhin2015, contreras2024}.  We note that a direct derivation of such a gluing formula is possible; however,  the QFT-based approach provides insight into why such relations are expected to hold.   This suggests that  if we can represent combinatorial quantities as quantum field theoretic object,  then,  an application of locality principle may provide new insight on such quantities. We think of this work as a toy model in that direction.

This paper is organized as follows.  We recall basic definitions and key results that will be used throughout the paper in Section \ref{Section: preliminaries}.   We present main theorems and corollaries in Section 3 and also provide proofs of corollaries. We study Dirichlet-to-Neumann map on a finite graph,  which plays a crucial role in the proof of main results,  in Section 4. We also briefly mention QFT tools relevant to this work in Section 4 and we prove our main theorems in Section 5. Finally, in Section 6, we provide examples of Dirichlet-to-Neumann maps and use gluing formula to give closed form formula for the number of spanning trees.

\section{Notation and Conventions} \label{Section: preliminaries}
In this section,  we briefly discuss basic definitions and some known results used throughout this work.  We refer to \cite{cioabua2009} for graph theory related definitions.

\begin{definition}[Pseudo-determinant]
    For a square matrix $A$, the pseudo-determinant of $A$ is denoted by $\det^{\prime}(A)$ and is defined as the product of non-zero eigenvalues. 
 \end{definition}

\begin{remark} When a square matrix $A$ is invertible, then 
   $\det^{\prime}(A) = \det(A)$. 
\end{remark}

\noindent
The following lemma is immediate. 
\begin{lemma} Let $A$ be a square matrix and $p(\lambda) = \det(A + \lambda I)$ be the characteristics polynomial of $A$. If $0$ is an eigenvalue of $A$ with multiplicity $k$, then
\[
{\det}^{\prime}(A) =\dfrac{1}{k!}\dfrac{d^k}{d\lambda^k}p(\lambda)\Big|_{\lambda=0}.
\]
\end{lemma}

  Let us fix some notations and conventions from graph theory.  Throughout this work, we assume that a graph is a \textit{simple graph}, i.e., the graph has neither any self loops nor multiple edges between the same pair of vertices and also connected. We will use $V(G)$ to denote the set of vertices of a graph $G$ and $E(G)$ to denote the set of edges of $G$. A \emph{tree} is a connected graph with no cycles and a forest is a graph with no cycles. A \emph{rooted tree} is a pair $(T,v)$ where $T$ is a tree and $v$ is a vertex in $T$ and a rooted forest is a forest with each of its components a rooted tree. Given a positive integer $k$, a $k$-forest is a forest with $k$ connected components. We say that a subgraph $Y$ of a graph $G$ is \emph{full subgraph} if $\{u,v\} \in E(G)$ and $u,v \in V(Y)$, then $\{u,v\} \in E(Y)$.

A \emph{weight} on a graph $G$ is a function $w: V(G) \times V(G) \to \mathbb{R}$ such that $w((v_i,v_j))= w((v_j,v_i))$ for all $v_i, v_j \in V(G)$ and $w((v_i,v_j))=0$ if $ \{v_i,v_j\}\notin E(G)$. A weight $w$ can be represented by the matrix $\displaystyle\boldsymbol{w} = (w_{ij})_{\{1\leq i,j\leq |V(G)|\}}$.  A \emph{weighted graph} is a pair $(G,\boldsymbol{w)}$.   Let $\mathcal{ST}(G)$ denote the set of all spanning trees in $G$. The spanning tree enumerator (or the Kirchhoff polynomial) \cite{Klee2019, Zhou2020} of the weighted graph $(G,\boldsymbol{w})$ is defined by as 
 \begin{equation*} 
     \tau(G, \boldsymbol{w}) = \sum_{T\in \mathcal{ST}(G)}\prod_{\{\{v_i,v_j\}\in E(T)\}} w_{ij}. 
 \end{equation*} The Laplacian matrix $L(G, \boldsymbol{w})$ of $(G,\boldsymbol{w})$ is given by

\begin{equation*}\label{Eqon: a weighted laplacian}
    L(G, \boldsymbol{w})(i,j) = \begin{cases}
        d (v_i, \boldsymbol{w}) \quad \text{if} \quad i=j, \\
        - w_{ij} \quad \quad\, \text{if} \quad i \ne j \,\,\text{and}\,\, \{v_i,v_j\}\in E(G), \\
        0 \qquad \quad\,\,\,\,\, \text{otherwise},
    \end{cases}
\end{equation*} where $d (v_i, \boldsymbol{w})$ denote the sum of all $w_{ij}$ for which $\{v_i,v_j\}\in E(G)$.   When $w_{ij}=1$ for all $\{v_i,v_j\} \in E(G)$, we will denote the Laplacian by $L(G)$.  Similarly, we will denote $\tau(G, \boldsymbol{w})$ by $\tau(G)$ and it counts the number of spanning trees on $G$.

A version of the celebrated Matrix-Tree theorem relates the spanning tree enumerator to the Laplacian:
\begin{theorem}[\cite{kirchhoff1847, Klee2019, Zhou2020}] 
\begin{equation}
 \tau(G, \boldsymbol{w}) = {\det}'(L(G, \boldsymbol{w}))/|V(G)|
 \end{equation}
\end{theorem}

The Matrix-Tree theorem has been generalized in many different directions. One of them, that will be used in this work, is the so-called all minor version of the Matrix-Tree theorem. Let $U \subseteq V(G)$  and $L(G\,|\,U)$ be the matrix obtained from $L(G)$ by removing the rows and columns associated to vertices in $U$. The following theorem, by Chaiken states that $\det(L(G\,|\,U))$ counts the number of rooted spanning $k$-forests in $G$ with roots in $U$. 

\begin{theorem}[\cite{chaiken1982}]\label{Thm:Chaiken} Let $G$ be a simple graph with $V(G) = \{v_1,\ldots, v_n\}$ and $U \subseteq V(G)$ such that $|U|=k$. Then, 
$\det(L(G\,|\,U))$ counts the number of rooted spanning $k$-forests that are rooted in $U$. 
\end{theorem}

\paragraph{Gluing two graphs} Let us introduce the notion of gluing two graphs.  Let $G_1$ and $G_2$ be two graphs and $Y$ be a {full subgraph} of $G_1$ and $G_2$. Then,  we can construct a new graph $G_1\cup_{Y} G_2$ by identifying the vertices and edges present in $Y$: $V(G_1 \cup_{Y} G_2) = V(G_1)\sqcup V(G_2)/\sim $ and $E(G_1 \cup_{Y} G_2) = E(G_1) \sqcup (E(G_2)/\sim $. Here, $v_1 \in V(G_1)$ is identified with $v_2 \in V(G_2)$, i.e, $v_1\sim v_2$ if they represent the same vertex in $Y$.  Similarly, $e_1 \in E(G_1)$ is identified with $e_2 \in E(G_2)$, i.e., $e_1\sim e_2$ if they represent the same edge in $Y$.

We end this section by recalling the notion of Schur complement of a set of vertices in a graph introduced in \cite{Zhou2020}.  Let $G$ be a graph and $V(G) = V_1 \sqcup V_2$.  Let us write 
\begin{equation}\label{eqon: block matrix of Laplacian}
    L(G) = 
\begin{bmatrix}
    A & B \\
    B^T & D 
\end{bmatrix}
\end{equation} where $A$ and $D$ are principal minors corresponding to $V_1$ and $V_2$ respectively.  Let $S = D - B^{T}A^{-1}B$ be the Schur complement of $A$ in $L(G)$. 

\begin{definition}
 Define a graph $S(V_2)$ by setting $V(S(V_2)) = V_2$ and there is an edge joining $v,v^{\prime} \in V_2, v \ne v'$ if and only if $vv^{\prime}$ component of $S$ is non-zero. The graph $S(V_2)$ is called the Schur complement  graph associated with $V_2$.     
\end{definition}

 Let $Y$ be a subgraph of $G$ and assume that $V(Y) = V_2$. Then, the matrix $D - B^{T}A^{-1}B$ is also called the \emph{Dirichlet-to-Neumann map} \cite{Curtis2000}. It is also called  the \emph{the response matrix} in Electrical Network Theory literature.  We will denote the Dirichlet-to-Neumann map by $\DN(G,Y)$. {This map is discussed in detail in Section \ref{Section: proof tools}.

\section{Main Results}

In this section, we state our main results and prove corollaries. 

\subsection{Gluing Formula for the Pseudo-Determinant}\label{subsection: gluing formula for pseudo-determinant}
The first main result is the gluing formula for the pseudo-determinant of the Laplacian on a finite simple graph. 
Let $G_1$ and $G_2$ be two graphs and $G = G_1 \cup_{Y} G_2$ where $Y$ is a full subgraph of $G_i,\, i=1,2$. For $i=1,2$, let 
 \[L(G_i) 
 = \begin{bmatrix}
      A_i& B_i \\[1mm]
         B_i^{T}& D_i
 \end{bmatrix}, 
\] and $\DN(G_i,Y)$ be the Dirichlet-to-Neumann map associated to $Y$ in $G_i$. 

\begin{theorem}\label{theorem:main theorem-1}The following relation holds: 
\begin{equation}\label{eqon: gluing formula Laplacian}
{\det}^{\prime}(L(G)) = \dfrac{|V(G)|}{|V(Y)|}\det(A_1) \det(A_2) {\det}^{\prime}(\DN(G_1,Y,G_2)),
\end{equation} where $\DN(G_1,Y,G_2) = \DN(G_1,Y) + \DN(G_2,Y) - L(Y)$. 
\end{theorem}

\subsection{Schur Complement Type Formulas}
Let $G$ be a graph and $Y$ be a full subgraph of $G$. Then, using the same argument used to prove Theorem \ref{theorem:main theorem-1}, we get the following Schur complement  formula for the pseudo-determinant of the Laplacian. 
\begin{theorem} \label{theorem:schur complement formula}
Let
\[
L(G) = \begin{bmatrix}
    A & B \\
    B^{T} & D 
\end{bmatrix}, 
\] then the following holds: 
\begin{equation}\label{eqon: Schur complement}
    {\det}'(L(G)) = \dfrac{|V(G)|}{|V(Y)|}\det(A){\det}'(\DN(G,Y)). 
\end{equation}
\end{theorem}

\begin{remark}
 In the light of Proposition \ref{prop: gluing DN maps},  it is possible to deduce Theorem \ref{theorem:main theorem-1} from Theorem \ref{theorem:schur complement formula}. If $G = G_1\cup_{Y}G_2$, then it can be shown that $\det(A) = \det(A_1)\det(A_2)$.  
\end{remark}

\noindent
 Next, we recall the Schur complement  type formula for the number of spanning trees $\tau(G)$ derived in \cite{Zhou2020} and deduce that it is equivalent to Theorem \ref{theorem:schur complement formula}.
Let $G$ be a graph with $V(G) = V_1 \sqcup V_2$ and $S(V_2)$  be the Schur complement graph associated to $V_2$ in $G$.  Let $S$ be the Schur-complement matrix introduced in Section \ref{Section: preliminaries}.  It can be shown (Proposition \ref{prop: DN is a weighted Laplacian}) that $S$ is a weighted graph Laplacian on $S(V_2)$ for the weight $\boldsymbol{w}$ induced by $S$. Let $\tau(S(V_2, \boldsymbol{w}))$ be the spanning tree enumerator of  $(S(V_2), \boldsymbol{w})$.  In \cite{Zhou2020}, it is proved that 
\begin{equation}\label{eqon: Schur complement formula for spanning trees}
    \tau(G) = {\det}^{\prime}(A)\, \tau(S(V_2), \boldsymbol{w})
\end{equation}

\noindent
Now, when $V_2 = V(Y)$, we find that Theorem \ref{theorem:schur complement formula} is equivalent to the relation (\ref{eqon: Schur complement formula for spanning trees}). More precisely, 

\begin{corollary} The following statements are equivalent: 
\begin{enumerate}
    \item[(a)] Schur complement  type formula holds for $\tau(G)$ \cite{Zhou2020}: 
    \[
\tau(G) = {\det}^{\prime}(A)\, \tau(S(V_2), \boldsymbol{w}). 
\]
\item[(b)] Schur complement  formula holds for the pseudo-determinant of Laplacian:
\[
{\det}'(L(G)) = \dfrac{|V(G)|}{|V(Y)|}\det(A)\,{\det}'(\DN(G,Y)). 
\]
\end{enumerate}
\begin{proof} The corollary immediately follows from ${\det}'(L(G)) =|V(G)| \tau(G)$ and \\ ${\det}'(\DN(G,Y)) = |V(Y)|\tau(S(V_2), \boldsymbol{w})$. 
    
\end{proof}
\end{corollary}

\subsection{Gluing Relation for Spanning Trees}
Let us consider the same set up and notations used in Subsection \ref{subsection: gluing formula for pseudo-determinant}.   
\begin{theorem} \label{cor: main cor-2}
 The number of spanning trees $\tau(G)$  satisfies the following gluing relation:
 \[
 \tau(G) = \tau(G_1) \tau(G_2)  C(G_1, G_2, Y)
 \] where $C(G_1, G_2, Y) = \dfrac{|V(Y)| \det^{\prime}(\DN(G_1,Y,G_2))}{\det^{\prime}(\DN(G_1,Y))\det^{\prime}(\DN(G_2,Y))}$.

 \end{theorem}

\begin{remark} After rewriting  
\[
C(G_1, G_2, Y) = \dfrac{\det^{\prime}(\DN(G_1,Y,G_2))}{|V(Y)|}\Bigg/ \left(\dfrac{\det^{\prime}(\DN(G_1,Y))}{|V(Y)|}  \dfrac{\det^{\prime}(\DN(G_2,Y))}{|V(Y)|}\right),
\] and using Schur complement type 
formula \cite{Zhou2020}, the quantity $C(G_1, G_2, Y)$ can be interpreted as the ratio of spanning tree enumerators. 
\end{remark}

It is also possible to express the gluing formula for the number of spanning trees in terms of the number of rooted spanning forests. Assume that $|V(Y)| = k$.  We recall from the all minor theorem of Chaiken \cite{chaiken1982}, Theorem \ref{Thm:Chaiken}, that $\det(A_i)= N(G_i,F,V(Y))$, where $N(G_i,F,V(Y))$ counts the number of spanning $k$-forests in $G_i$ that are rooted in  $V(Y)$. Then,  using this interpretation of the determinant of minors,  we immediately have the following corollary which gives a gluing formula for the spanning trees in terms of the spanning forests and the determinant of the Dirichlet-to-Neumann map.  
\begin{corollary}
    \label{cor: main cor-1}The following relation holds:
\[ 
     \tau(G)=N(G_1,F,V(Y))\, N(G_2,F,V(Y))\dfrac{{\det}^{\prime}(\DN(G_1,Y,G_2))}{|V(Y)|}.
\]   
\begin{proof} The corollary follows if we use 
\[
{\det}^{\prime}(L(G))= {|V(G)|}\tau(G)
\] and \[\det(A_i)= N(G_i,F,Y)\] in Theorem \ref{theorem:main theorem-1}.  
\end{proof}    
\end{corollary}

\noindent
Since the matrix $\DN(G_1,Y,G_2)$ is a weighted Laplacian on the Schur complement  graph $S(V(Y))$ with respect to a weight $\boldsymbol{w}$ (Proposition \ref{prop: DN is a weighted Laplacian}),  we can rephrase Corollary \ref{cor: main cor-1} completely in terms of the combinatorial quantities: 

\begin{corollary}  The number of spanning trees $\tau(G)$  satisfies the following gluing relation:
    \[
 \tau(G)= N(G_1,F,V(Y))N(G_2,F,V(Y))\tau(S(V_2), \boldsymbol{w}). 
    \]
\end{corollary}

\subsection{Gluing Relation for Spanning Forests}

It is known that $\det(L(G)+I)$ is the number of all rooted spanning forests in $G$ \cite{chebotarev2006,  DD1988, Knill2014}. 

We can use the gluing formula for $\det(L(G)+ I)$ \cite{reshetikhin2015, contreras2024}, to give a gluing relation for the number of all rooted spanning forests. Let us consider the same set up and notations as in Subsection \ref{subsection: gluing formula for pseudo-determinant}. 

\begin{theorem}\label{theorem: main theorem-2}
    Let $G_1$ and $G_2$ be two graphs and $G = G_1 \cup_{Y} G_2$, where $Y$ is a full subgraph of $G_i,\, i=1,2$. Then, 
    \[
    N(G,F) = N(G_1,F) N(G_2,F) \dfrac{\det(\DN(G_1,Y,G_2,1))}{\det(\DN(G_1,Y,1))\det(\DN(G_2,Y,1))}
    \] where $\DN(G_1,Y,G_2,1)$, $\DN(G_1,Y,1)$ and $\DN(G_2,Y,1)$ are Dirichlet-to-Neumann maps defined in Section \ref{subsubsec:DtoN}.
\end{theorem}

\section{Auxiliary Results and Tools}\label{Section: proof tools}
In this section,  we discuss auxiliary results and tools needed for proofs of main results.  Since the Dirichlet-to-Neumann map plays an important role in the proof,  we discuss this map in detail, which can be of independent interest on its own. 

\subsection{Dirichlet-to-Neumann Map}\label{subsubsec:DtoN} 
Let $G$ be a graph with $V(G) = \{v_1,\ldots,v_r, v_{r+1},\ldots v_n\}$. Let $F_{G}$ denote the set of 
real valued functions on $V(G)$. Thus, it is the vector space of dimension $|V(G)|$. Let $Y$ be a full subgraph of $G$ and assume that $V(Y) =\{ v_{r+1}, \ldots, v_{r+k}\}$ with $n=r+k$. Let $C(G,\lambda)= L(G) + \lambda I$ be the characteristics matrix of $L(G)$. The graph version of the Dirichlet Boundary Value Problem (DBVP) for $C(G,\lambda)$ can be formulated as follows. For a fixed $\eta\in F_{Y}$, the problem is to find $\phi \in F_{G}$ such that 
\begin{eqnarray} 
\phi(v)&=& \eta(v) \;\;\mr{for\;all}\; v\in V(Y),\label{eqon:Dirichlet eq1}
\\ \label{eqon:Dirichlet eq2}
C(G,\lambda) (\phi)(v)&=&0 \;\;\mr{for\;all}\; v\in V(G)\setminus V(Y).
\end{eqnarray} 
It is well known that for $\lambda\geq 0$, the solution to DBVP (\ref{eqon:Dirichlet eq1})--(\ref{eqon:Dirichlet eq2}) exists and is unique \cite{Curtis2000}. In fact, the uniqueness of the solution can be used to construct the solution as follows. Let us identify $\phi \in F_{G}$ with the vector $\boldsymbol{\phi}= [\phi_1, \ldots, \phi_n]^{T}$, and $\eta\in F_{Y}$ with the vector $\boldsymbol{\eta}=[\eta_{r+1}, \ldots , \eta_n]^{T}$. Let us write the $C(G,\lambda)$ as a block matrix according to vertices of $G$ not belonging to $Y$ (``bulk vertices'') and belonging to $Y$ (``boundary vertices''):
\begin{equation}\label{block matrix represetation of characteristic matrix}
    C(G,\lambda) = \begin{bmatrix} 
    A(\lambda) & B \\[1mm] 
     B^{T} & D(\lambda)
   \end{bmatrix}
\end{equation} 

\noindent
From the uniqueness of the solution to (\ref{eqon:Dirichlet eq1})--(\ref{eqon:Dirichlet eq2}), it follows that $A(\lambda)$ is invertible. Let
$\boldsymbol{\phi}=
\begin{bmatrix}
-A(\lambda)^{-1}B\boldsymbol{\eta} &
\boldsymbol{\eta}
\end{bmatrix}^{T}$ 
Then, 
\begin{equation}\label{eqon:DN map}
    C(G,\lambda)\boldsymbol{\phi} = \begin{bmatrix}
\boldsymbol{0}\\
(D-B^{T}A(\lambda)^{-1}B)\boldsymbol{\eta}
\end{bmatrix}.
\end{equation} Let $\phi$ be the function represented by the vector $\boldsymbol{\phi}$, then we have just shown that $\phi$ is indeed the solution of the DBVP (\ref{eqon:Dirichlet eq1})--(\ref{eqon:Dirichlet eq2}). 

\begin{remark}
    For the existence and uniqueness of the solution of the system DBVP (\ref{eqon:Dirichlet eq1})--(\ref{eqon:Dirichlet eq2}) it is sufficient that $A(\lambda)$ is invertible and this is the case as long as $\lambda$ is not an eigenvalue of $A$. 
\end{remark}

From equation (\ref{eqon:DN map}),  it follows that any $\eta \in F_Y$ can be mapped to an element in $F_Y$ represented by $(D-B^{T}A(\lambda)^{-1}B)\boldsymbol{\eta}$. This allows us to define the Dirichlet-to-Neumann map. 

\begin{definition}[Dirichlet-to-Neumann Map] Assume that $A(\lambda)$ is invertible.  The map 
 \begin{equation}\label{eqon:DN Map}
\mathrm{DN}(G,Y,\lambda): F_Y \to F_Y, \quad \eta \mapsto \eta^{\prime},
\end{equation}
where $\eta^{\prime}\in F_Y$ representing $(D-B^{T}A(\lambda)^{-1}B)\boldsymbol{\eta}$ is called the \emph{Dirichlet-to-Neumann} map of $Y$  in $G$ associated with $C(G,\lambda)$.  When $\lambda = 0$, we denote the Dirichlet-to-Neumann map by simply $\mathrm{DN}(G,Y)$. 
\end{definition}

The matrix representation of $\mathrm{DN}(G,Y,\lambda)$ is given by $D-B^{T}A(\lambda)^{-1}B$.  By abusing the notation,  we will 
denote $D-B^{T}A(\lambda)^{-1}B$ by $\mathrm{DN}(G,Y,\lambda)$. 
Note that $D-B^{T}A^{-1}B$ is the Schur complement of $A$ in $C(G,\lambda)$.

Assume that $\lambda>0$. Then, $C(G,\lambda)$ is invertible and it follows that $\mathrm{DN}(G,Y,\lambda)$ is also invertible \cite{Horn2012}. Let, 
\[
C(G,\lambda)^{-1} = 
    \begin{bmatrix}
        \widehat{A}& \widehat{B} \\[1mm]
         \widehat{B}^{T}& \widehat{D}
    \end{bmatrix}
\]
Using the Schur complement formula to compute the inverse of a $2\times 2$ block matrix \cite{hager1989, Horn2012}, we have the following statement.
\begin{lemma} Assume that $\lambda>0$. Then, the component $\hat{D}$ is the inverse of Dirichlet-to-Neumann map: 
\begin{equation} \label{eqn: inverse of dirichlet-to-neumann}
\widehat{D} = \mathrm{DN}(G,Y,\lambda)^{-1}
\end{equation}
\end{lemma}

\begin{remark}
    In \cite{contreras2024}, the equation (\ref{eqn: inverse of dirichlet-to-neumann}) has been used to define the Dirichlet-to-Neumann map on a graph, i.e., $\mathrm{DN}(G,Y,\lambda) = \widehat{D}^{-1}$. 
\end{remark}

\noindent
Let us explore some properties of Dirichlet-to-Neumann map. Consider the matrix $A=A(\lambda)$ that appeared in the block-decomposition of $C(G,\lambda)$. For $\lambda$ in a small neighborhood of $0$, $A(\lambda)$ is invertible. From this observation, we can immediately show the following. 

\begin{lemma} \label{lemma:continuity of DN map}
The matrix valued function 
$\lambda \mapsto \mathrm{DN}(G,Y,\lambda)$ 
is continuous in a neighborhood of $0$. In particular, 
\[
\mathrm{DN}(G,Y) = \lim\limits_{\lambda\to 0}\,\,\mathrm{DN}(G,Y,\lambda).
\]
\end{lemma}

\noindent
Let us recall the following well-known result relating the rank of $C(G,\lambda)$, rank of $A$, and $\DN(G,Y,\lambda)$. 
\begin{lemma}[\cite{Matsaglia1974}] \label{lemma:rank additivity}
The following holds:
    \begin{equation} \label{eqon: rank equation}
        \mathrm{rank}(C(G,\lambda)) = \mathrm{rank}(A(\lambda)) + \mathrm{rank}(\DN(G,Y,\lambda)) . 
    \end{equation}
\end{lemma}

\begin{corollary} \label{corollary: kernel of DN map}
The null space of $\DN(G,Y)$ is one dimensional. Consequently, the restriction map $R: F_G \to F_Y$,  $\phi \mapsto \phi|_{Y}$ induces a vector space isomorphism from the null space of $L(G)$ to the null space of $\DN(G,Y)$. 
\begin{proof} Assume that $F_G$ is $n$-dimensional and $F_Y$ is $k$-dimensional. 
Recall that $C(G,\lambda)\big|_{\lambda =0}$ is the Laplacian on $G$ and hence it has one dimensional null space. Thus, its rank is $n-1$.

\noindent
Since $A$ is invertible, $\mathrm{rank}(A) = n-k$.  From Lemma \ref{lemma:rank additivity}, it follows that $\mathrm{rank}(\DN(G,Y)) = k-1$. This means the null space of $\DN(G,Y)$ is one dimensional.  

\noindent
We know that the null space of $L(G)$ consists of the constant functions on $V(G)$ i.e. the null space is $\mathrm{Span}\{\boldsymbol{1}\}$, where $\boldsymbol{1}$ is the constant function that assigns $1$ to each vertex of $G$. Now, $L(G)\boldsymbol{1} = \boldsymbol{0}$ implies that $\DN(G,Y)R(\boldsymbol{1}) =\boldsymbol{0}$. This means the null space of $\DN(G,Y)$ contains constant functions on $Y$. Thus, by the previous paragraph, the null space  of $\DN(G,Y)$ consists of constant functions. Moreover, the map $R: F_G \to F_Y$ induces a non-trivial linear map from the null space of $L(G)$ to the null space of $\DN(G,Y)$. Since both spaces are one dimensional, this is a vector space isomorphism.  
\end{proof}
\end{corollary}

It turns out that $\DN(G,Y)$ is a weighted Laplacian \cite{Zhou2020, Devriendt2022}.  For the sake of completeness,  we sketch a proof of this in the following proposition.

\begin{proposition}\label{prop: DN is a weighted Laplacian}
The map $\DN(G,Y)$ is a weighted graph Laplacian on $S(V(Y))$, where $S(V(Y))$ is the Schur complement  graph induced by $V(Y)$. 
\begin{proof}
Since $A$ is positive definite and $L(G)$ is positive semi-definite,  it follows that $\DN(G,Y)$ is positive semi-definite \cite{Gallier2011}.  From corollary  \ref{corollary: kernel of DN map},  it follows that the Null space of $\DN(G,Y)$ is spanned by the vector $\boldsymbol{1}$. To show it is indeed a Laplacian, we only need to show that the off diagonal entries of $\DN(G,Y)$ are non-positive.  Using the sequential property of Schur complement \cite{zhang2006}, it suffices to assume $|V(Y)|= |V(G)|-1$. In this case,  it can be checked that the off diagonal entries are non-positive $\DN(G,Y)$ by a direct computation \cite{Devriendt2022}. This proves that $\DN(G,Y)$ is indeed a weighted graph Laplacian on $S(V(Y))$. 
\end{proof}
\end{proposition}

\subsubsection{Gluing formula for Dirichlet-to-Neumann maps} \label{subsection: gluing DN maps}
In this sub-section,  we establish  a gluing relation for the Dirichlet-to-Neumann map.  Let $G_1$ and $G_2$ be two graphs and $G = G_1 \cup_{Y} G_2$, where $Y$ is a full subgraph of $G_i,\, i=1,2$. Then, we have three Dirichlet-to-Neumann maps: $\DN(G,Y,\lambda)$, $\DN(G_1,Y,\lambda)$, and $\DN(G_2,Y,\lambda)$. The  gluing relation between these Dirichlet-to-Neumann maps are discussed in Proposition \ref{prop: gluing DN maps}.  The relation (\ref{eqon: gluing formula for the massive dirichlet-to-neumann map}) is established in \cite{contreras2024},  but the proof given here is different and direct; and the relation (\ref{eqon: gluing formula for the dirichlet-to-neumann map}) immediately follows from (\ref{eqon: gluing formula for the massive dirichlet-to-neumann map}).

\begin{proposition}[Gluing formula for the Dirichlet-to-Neumann map] \label{prop: gluing DN maps}
Assume that $\lambda >0$ or $\lambda$ is in a small neighborhood of $0$.  Let us define $\DN(G_1,Y,G_2,\lambda)=\DN(G_1,Y,\lambda)+\DN(G_2,Y,\lambda)- C(Y,\lambda). $ Then, the following holds: 
 \begin{equation}\label{eqon: gluing formula for the massive dirichlet-to-neumann map}
 \DN(G,Y,\lambda) =  \DN(G_1,Y, G_2,\lambda).
 \end{equation}
 In particular, by letting $\lambda\to 0$,  we have:
  \begin{equation}\label{eqon: gluing formula for the dirichlet-to-neumann map}
  \DN(G,Y) = \DN(G_1,Y, G_2).
 \end{equation}
 \begin{proof} The proof immediately follows from identifying the relation between $C(G,\lambda)$, $C(G_i,\lambda)$ along $Y$. More precisely, let 
 \[C(G_i,\lambda) 
 = \begin{bmatrix}
      A_i(\lambda) & B_i \\[1mm]
         B_i^{T} &  D_i(\lambda)
 \end{bmatrix}.
\] 
Similarly, let 
 \[
 C(G,\lambda) 
 = \begin{bmatrix}
      A(\lambda) & B \\[1mm]
         B^{T}& D(\lambda)
 \end{bmatrix}. 
 \] 
Then,  we can check that
\[
A(\lambda) = 
\begin{bmatrix} 
A_1(\lambda) & 0 \\
0 & A_2(\lambda)
\end{bmatrix}, \quad
 B = 
 \begin{bmatrix}
 B_1 \\ B_2
\end{bmatrix}, 
\] and 
  \begin{equation*}
      D(\lambda) = D_1(\lambda) + D_2(\lambda) - C(Y,\lambda). 
  \end{equation*} Hence,
 \begin{align*}
 \DN(G,Y,\lambda)& = D(\lambda) - B^{T}A(\lambda)^{-1}B\\
 & = [D_1(\lambda)-B_1^{T}A_1(\lambda)^{-1}B_1] +  [D_2(\lambda)-B_2^{T}A_2(\lambda)^{-1}B_2] - C(Y,\lambda)\\
 & = \DN(G_1,Y,\lambda)+\DN(G_2,Y,\lambda)- C(Y,\lambda).
 \end{align*} 
  
 \end{proof} 
\end{proposition}

\subsubsection{Pseudo-determinant of the Dirichlet-to-Neumann Map} 
Let us consider the same set up as in Subsection \ref{subsection: gluing DN maps}.  It is not true that $\det(\DN(G_1,Y,G_2,\lambda)) \ne \det(\DN(G_1,Y,G_2) + \lambda I)$ in general.   But it is still possible to relate the derivative of $\det(\DN(G_1,Y,G_2,\lambda)) $ with respect to $\lambda$ at $\lambda =0$ to the pseudo-determinant  of $\DN(G_1,Y,G_2)$.

We next prove the following result which plays the key role in the proof of main theorems.  This theorem can itself be of independent interest. 
\begin{theorem} \label{Thm: pesudo-determinant of DN Map}
The following holds: 
\[
 \dfrac{d}{d\lambda}\det(\DN(G_1,Y,G_2,\lambda))\bigg|_{\lambda=0} = \dfrac{|V(G)|}{|V(Y)|} {\det}^{\prime}(\DN(G_1,Y,G_2))
\]
 \end{theorem}

\noindent
A proof of Theorem \ref{Thm: pesudo-determinant of DN Map},  given here,  uses the perturbation theory of Laplacian \cite{kato2013}.  Since the computation of the derivative is local,  it suffices to assume $\lambda$ is in a small neighborhood of $0$.  This proof is inspired by the proof of BFK gluing formula for the zeta regularized determinant of the Laplacian on compact Riemannian manifolds \cite{burghelea1992}. 

We will first prove a few lemmas that are needed to prove this theorem. 

\begin{lemma}\label{lemma: Representation of inverse of DN map}
Assume that $\lambda \ne 0$ such that $C(G,\lambda)$ is invertible and $|V(G)|=n$. Let $\psi_1, \ldots, \psi_{n}$ be an orthonormal eigenbasis of $L(G)$ associated with the eigenvalues $0 = \sigma_1< \sigma_2< \ldots < \sigma_n$. Then,
\begin{equation}
    \DN(G_1,Y,G_2,\lambda)^{-1}\eta = \dfrac{1}{\lambda} \langle\eta\,,\, R(\psi_1)\rangle R(\psi_1)+ \sum_{j=2}^{n} \dfrac{1}{\sigma_j+\lambda}\langle\eta\,,\, R(\psi_j)\rangle R(\psi_j)
\end{equation}
    \begin{proof} From the assumption, it follows that the eigenvalues of $C(G,\lambda)^{-1}$ are $1/\lambda, 1/(\sigma_2+\lambda), \ldots, 1/(\sigma_n + \lambda)$ with orthonormal eigenbasis $\psi_1,\ldots, \psi_n$. This means for any $\psi \in F_G$, we have 
    \[
    C(G,\lambda)^{-1}\psi = \dfrac{1}{\lambda} \langle\psi\,, \psi_1 \rangle \psi_1  + \sum_{j=2}^{n} \dfrac{1}{\sigma_j + \lambda} \langle\psi\,, \psi_j \rangle \psi_j
    \]
Now,  the desired result follows from relation (\ref{eqn: inverse of dirichlet-to-neumann}) and Proposition \ref{prop: gluing DN maps}.  
    \end{proof}
\end{lemma}

\noindent
Let $\psi_1, \ldots, \psi_{n}$ and $0 = \sigma_1< \sigma_2< \ldots < \sigma_n$ as in Lemma \ref{lemma: Representation of inverse of DN map}. For fixed $\lambda\ne 0$, let us define $I: F_Y \to F_Y$ and $I^{\prime}: F_Y \to F_Y$ by 

\begin{align}
    \begin{array}{lll}
        I(\eta) = \dfrac{1}{\lambda} \langle\eta\,,\, R(\psi_1)\rangle R(\psi_1) \quad \text{and} \quad 
        I^{\prime}(\eta) = \sum\limits_{j=2}^{n} \dfrac{1}{\sigma_j+\lambda}\langle\eta\,,\, R(\psi_j)\rangle R(\psi_j).
    \end{array}
\end{align} Then for any fixed $\eta\in F_Y$,  $I^{\prime}(\eta) \quad \text{is bounded as} \quad \lambda \to 0$.   Furthermore,  from Lemma \ref{lemma: Representation of inverse of DN map}, it follows that 
\begin{equation}
    \DN(G_1,Y,G_2,\lambda)^{-1} = \dfrac{1}{\lambda} I + I^{\prime}
\end{equation}

\noindent Assume that $|V(Y)| =k$ and let $\mu_1(\lambda) < \mu_2(\lambda) < \ldots < \mu_k(\lambda)$ be eigenvalues and $\phi_1(\lambda), \ldots,\phi_k(\lambda)$ be corresponding orthonormal eigen basis of $\DN(G_1,Y,G_2,\lambda)$.  Let 
\begin{equation} \mu_j = \mu_{j}(\lambda)|_{\lambda=0}\quad \text{and} \quad \phi_j = \phi_j(\lambda)|_{\lambda=0}. 
\end{equation} Then by the perturbation theory of eigenvalues and eigenfunctions \cite{kato2013}, we have the following.

\begin{lemma} The eigenvalues and a corresponding orthonormal eigen basis of $\DN(G_1,Y,G_2)$ are given by $0=\mu_1 <\mu_2 <\ldots < \mu_k$ and $\phi_1,\ldots, \phi_k $ 
 respectively.    
\end{lemma}

The following gives information about the behavior of $\mu_1(\lambda)$ as $\lambda \to 0$. 

\begin{lemma} \label{lemma: behavior of perturbation of lowest eigenvalue}
Let $a_1(\lambda) = \langle R(\psi_1), \phi_1(\lambda)\rangle$. Then,
\begin{equation}
    \dfrac{1}{\mu_1(\lambda)} = \dfrac{1}{\lambda} a_1(\lambda)^2 + b_1(\lambda), 
\end{equation} where $b_1(\lambda)$ is bounded as $\lambda \to 0$. Consequently,
\begin{equation}
    \mu_1(\lambda) = \dfrac{\lambda}{a_1(\lambda)^2} + \lambda^2 b_2(\lambda) 
\end{equation} as $\lambda \to 0$ such that $b_2(\lambda)$ is bounded as $\lambda \to 0$.
\begin{proof} We know that 
\[
 \dfrac{1}{\mu_1(\lambda)} = \langle \DN(G_1,Y,G_2,\lambda)^{-1} \phi_1(\lambda), \phi_1(\lambda)\rangle
\] and 
\[
\DN(G_1,Y,G_2,\lambda)^{-1} \phi_1(\lambda) = \dfrac{1}{\lambda} I(\phi_1(\lambda)) + G(\phi_1(\lambda)).  
\]Now using,
\[
I(\phi_1(\lambda)) = a_1(\lambda) R(\psi_1)
\] we get 
\begin{align*}
\dfrac{1}{\mu_1(\lambda)}& = \langle \DN(G_1,Y,G_2,\lambda)^{-1} \phi_1(\lambda), \phi_1(\lambda)\rangle \\
& = \dfrac{1}{\lambda} a_1(\lambda) \langle R(\psi_1, \phi_1(\lambda)\rangle + b_1(\lambda) \\
& = \dfrac{a_1(\lambda)^2}{\lambda} + b_1(\lambda),  
\end{align*} where $b_1(\lambda) = \langle I^{\prime}(\phi_1(\lambda), \phi_1(\lambda) \rangle$. Now, $b_1(\lambda)$ is bounded as $\lambda\to 0$ immediately follows from $I^{\prime}$ is bounded  and $\phi_i(\lambda) \to \phi_1$ as $\lambda \to 0$. The second statement immediately follows from this. 
\end{proof}
\end{lemma}

\noindent
Next we compute $a_1 = a_1(\lambda)|_{\lambda =0}$ explicitly. 

\begin{lemma} The following holds: 
\begin{equation}
    a_1 = \sqrt{\dfrac{|V(Y)|}{|V(G)|}}. 
\end{equation} 
    \begin{proof} We recall from Corollary \ref{corollary: kernel of DN map}, that the restriction map $R$ induces an isomorphism from the null space of $L(G)$ to the null space of $\DN(G_1,Y,G_2)$. Let $\boldsymbol{1}$ be the constant function on $G$ that assigns to each vertex $1$, $\psi_1 = \left\{\dfrac{1}{\sqrt{|V(G)|}}\boldsymbol{1}\right\},$
    $\phi_1=R(\boldsymbol{1})/\sqrt{|V(Y)|}$. Then, $\left\{\psi_1\right\}$ is a basis of the null space of $L(G)$ such that $\langle\psi_1\,, \psi_1\rangle =1$ and $\left\{\phi_1\right\}$ is a basis of the null space of $\DN(G_1,Y,G_2)$ such that $\langle\phi_1\,, \phi_1\rangle =1$. Furthermore, $R(\psi_1) = \frac{\sqrt{|V(Y)|}}{\sqrt{|V(G)|}}\phi_1$
    
    Thus,
    \[
    a_1 = \langle \phi_1\,,\, R(\psi_1)\rangle = \sqrt{\dfrac{|V(Y)|}{|V(G)|}} \langle \phi_1\,,\, \phi_1\rangle = \sqrt{\dfrac{|V(Y)|}{|V(G)|}}. 
    \]
    \end{proof}
\end{lemma}

\begin{proof}[Proof of Theorem \ref{Thm: pesudo-determinant of DN Map}] From Lemma \ref{lemma: behavior of perturbation of lowest eigenvalue}, it follows that 
\[
\det(DN(G_1,Y,G_2,\lambda)) = \prod_{j=1}^{k} \mu_j(\lambda) = \dfrac{\lambda}{a_1(\lambda)^2}\prod_{j=2}^k \mu_j(\lambda) + \lambda^2 b_3(\lambda)
\] such that $b_3(\lambda)$ is bounded as $\lambda\to 0$. This implies that 
\[
 \dfrac{d}{d\lambda}\det(\DN(G_1,Y,G_2,\lambda))\bigg|_{\lambda=0} = \dfrac{1}{a_1^2} \prod_{j=2}^{k}\mu_j = \dfrac{|V(G)|}{|V(Y)|}{\det}^{\prime}(\DN(G_1,Y,G_2)),
\]
 as needed.   
\end{proof}

\subsection{The Locality Principle and Gluing Formula for Determinant}

The main objective of this sub-section is to briefly discuss the locality principle in QFT and how it can be used to derive the gluing formula for determinant of the massive Laplacian $L(G)+m^2$ on a graph $G$. These ideas and results are discussed in \cite{reshetikhin2015, contreras2024} and references therein.

 \subsubsection{The Locality Principle in QFT} 
 
The locality principle in QFT states that a QFT on a space-time is determined by its structure at ``short distances''.   In the path integral approach to QFT,  the key object of interest is the so-called  partition function.  Given a space-time,  a partition function is an integral defined on the space of fields on the space-time.  The locality principle,  in this context,  means if the space-time is cut into smaller pieces,  then the partition on the space-time can be recovered from the partition functions on the smaller pieces. 

The main message here is that if we replace the space time by a graph and we are able to relate a combinatorial object to the partition function of a QFT on a graph,  then,  exploiting the locality principle leads to a gluing formula for the combinatorial object.

\subsubsection{Gluing Formula for the Determinant of Massive Laplacian from the Locality Principle} 
  
Here,  we consider a free massive scalar field theory on a finite graph and briefly illustrate how the locality principle leads to a gluing formula for the determinant of massive Laplacian.  

\paragraph{Partition Function of a Free Massive Scalar Field theory.} Let $G$ be a finite graph.  Then,  for the scalar field theory,  the space of fields is given by $F_G$.  Let $m>0$ and $S_G: F_G \to \mathbb{R}$ be defined by 
\begin{equation}
    S_G(\phi)=\frac{1}{2}\langle\phi,C(G,m^2)\phi\rangle, 
\end{equation}  where $\langle\,, \rangle$ is the canonical inner product on $F_G$ defined by
\begin{equation}
   \langle \phi\,,\,\phi'\rangle = \sum\limits_{v\in V(G)}\phi(v)\phi'(v)
\end{equation}

For the so-called free massive scalar field theory,  the partition function is given by 
 \begin{equation}\label{partition function: closed graph}
Z(G) = \int_{F_G}D\phi\,\, e^{-S_G(\phi)}
\end{equation} where $\D\phi = \prod\limits_{v\in V(G)}\dfrac{d\phi(v)}{\sqrt{2\pi}}$.  Since the integral in (\ref{partition function: closed graph}) is a Gaussian integral,  we have 
\[
Z(G) = \det(C(G,m^2))^{-1/2}.
\]
More generally,  if $Y$ is a subgraph of $G$ and $\eta \in F_Y$,  we define
\begin{equation} \label{partition function: boundary}
    Z(G,Y,\eta) = \int_{\{\phi \in F_G\,:\, \phi|_Y = \eta\}} D\phi\,\, e^{-S_G(\phi)}
\end{equation} where $D\phi = \prod\limits_{v\in V(G)\backslash V(Y)}\dfrac{d\phi(v)}{\sqrt{2\pi}}$.   Let 
\[C(G,m^2) 
 = \begin{bmatrix}
      A(m^2) & B \\[1mm]
         B^{T}& D(m^2)
 \end{bmatrix}. 
 \] Then,  it can be shown that \cite{reshetikhin2015, contreras2024}
\begin{equation}\label{equation: partition function in the presence of boundary}
Z(G,Y,\eta) = \det(A(m^2))^{-1/2}e^{-\frac{1}{2}\langle \DN(G,Y,m^2)\eta, \eta\rangle}. 
\end{equation} 

\paragraph{Schur-Complement formula from Fubini's theorem} 
Using the Fubini's theorem \cite{folland1999} and relation (\ref{equation: partition function in the presence of boundary}), we can immediately establish the Schur complement formula for the determinant of $C(G,m^2)$. 
\begin{proposition} \label{remark: Schur complement formula for determinant of massive Laplacian} The following holds. 
\[
\det(C(G,m^2)) = \det(A(m^2)) \det(\DN(G,Y,m^2)).
\] 
 \begin{proof} Using Fubini's theorem 
 \begin{align*}
    \det(C(G,m^2))^{-1/2}&= Z(G) \\
    & = \int_{F_Y}D\eta Z(G,Y,\eta) \\
        & = \det(A(m^2))^{-1/2}\int_{F_Y}D\eta e^{-\frac{1}{2}\langle \DN(G,Y,m^2)\eta, \eta\rangle}\\
        & = \det(A(m^2))^{-1/2} \det(\DN(G,Y,m^2))^{-1/2}
 \end{align*} Thus,
 \[
\det(C(G,m^2)) = \det(A(m^2)) \det(\DN(G,Y,m^2)).
\] 
     
 \end{proof}   
\end{proposition}

 \paragraph{Gluing Formula for the Determinant of a Massive Laplacian}  Let $G_1$ and $G_2$ be two graphs and $G = G_1 \cup_{Y} G_2$, where $Y$ is a full subgraph of $G_i,\, i=1,2$.  Using the locality principle it can be shown that the following holds \cite{reshetikhin2015, contreras2024}:
  \begin{equation}\label{equation: formal fubini}
 Z(G)=\int_{F_Y}D\eta e^{S_Y(\eta)}Z(G_1,Y,\eta)Z(G_2,Y,\eta).
 \end{equation}
Using notations from Subsection  \ref{subsection: gluing DN maps},  the relation (\ref{equation: formal fubini}) implies the following gluing formula for the determinant of the massive Laplacian  \cite{reshetikhin2015, contreras2024}:
\begin{equation}\label{equation: gluing formula determinant of massive Laplacian}
    \det(C(G,m^2)) = \det(A(G_1,m^2))\det(A(G_2,m^2))\det(\DN(G_1,Y,G_2,m^2)).
\end{equation}

\section{Proof of Main Theorems} 

In this section, we present proofs of the main theorems. 

\subsection{Proof of Theorem \ref{theorem:main theorem-1} and Theorem \ref{theorem: main theorem-2}}\label{subsection: main proofs part-I}
 We will use the same set up as Subsection \ref{subsection: gluing formula for pseudo-determinant} and use the same notation as in \ref{subsection: gluing DN maps}. The key idea of the proof is to use the relation (\ref{equation: gluing formula determinant of massive Laplacian}).  We first prove Theorem \ref{theorem: main theorem-2}. 

\begin{proof}[Proof of Theorem \ref{theorem: main theorem-2}]
Taking $m =1 $ in equation (\ref{equation: gluing formula determinant of massive Laplacian}), we get 
\begin{equation}
    \det(C(G,1)) = \det(A(G_1,Y,1))\det(A(G_2,Y,1))\det(\DN(G_1,Y,G_2,1))
\end{equation}
It is known \cite{chebotarev2006,DD1988} that 
\[
\det(C(G,1)) = N(G,F),
\]
Now, the proof can be completed by using Remark \ref{remark: Schur complement formula for determinant of massive Laplacian}, from which it follows that
\[
\det(A(G_i,Y,1)) = \dfrac{\det(C(G_i,1))}{\det(\DN(G_i,Y,1))}= \dfrac{N({G}_i, F)}{\det(\DN(G_i,Y,1))} 
\]

\end{proof}

 Next we  prove Theorem \ref{theorem:main theorem-1}.  As noted earlier,  the computation of derivative is a local computation.  For that reason,  it suffices to concentrate on a small neighborhood of $0 \in \mathbb{C}$.  We note that for a given square matrix $M$, the function defined on $\mathbb{C}$ by $z\mapsto \det (M + zI)$ is a polynomial function on $\mathbb{C}$.  Thus, the functions $\det(C(G,z))$, $\det(A(G_1,z))$, and $\det(A(G_2,z))$,  are holomorphic functions on $\mathbb{C}$. Note that $A(G_1,z)$ and $A(G_2, z)$ are invertible for all $z$ in a neighborhood of $0\in \mathbb{C}$.   This implies that $z \mapsto \det(\DN(G_1,Y,G_2,z))$ is holomorphic function on a neighborhood of $0$ in $\mathbb{C}$.

\begin{proposition}\label{prop: gluing formula for complex parameter} For all $z$ in a small neighborhood of $0$ in $\mathbb{C}$,  we have 
\begin{equation}\label{eqon: gluing formula complex domain}
     \det(C(G,z)) = \det(A(G_1,Y,z))\det(A(G_2,Y,z))\det(\DN(G_1,Y,G_2,z))
\end{equation} 
\begin{proof} 
We know from (\ref{equation: gluing formula determinant of massive Laplacian}) that the equation (\ref{eqon: gluing formula complex domain}) holds for all $\lambda > 0$. By the identity theorem for holomorphic functions \cite{Conway1978}, it follows that the equation holds for all $z$ in a small neighborhood of $0$ in $\mathbb{C}$. 
\end{proof}
\end{proposition}

\begin{proof}[Proof of Theorem \ref{theorem:main theorem-1}]
We recall that 
\[
{\det}'(L(G)) = \frac{d}{d\lambda}\det(C(G,\lambda))\big|_{\lambda=0}\] By  (\ref{eqon: gluing formula complex domain}), we have
\begin{align*}
\begin{array}{lll}
        \frac{d}{d\lambda}\det(C(G,\lambda))\!\!\!&=\!\frac{d}{d\lambda}\det(A(G_1,Y,\lambda))\det(A(G_2,Y,\lambda))\det(\DN(G_1,Y,G_2,\lambda))\\\\
        &+\det(A(G_1,Y,\lambda))\dfrac{d}{d\lambda}\det(A(G_2,Y,\lambda))\det(\DN(G_1,Y,G_2,\lambda))\\\\
        &+\det(A(G_1,Y,\lambda))\det(A(G_2,Y,\lambda))\dfrac{d}{d\lambda}\det(\DN(G_1,Y,G_2,\lambda)).
\end{array}  
\end{align*}

\noindent
Setting $\lambda=0$,  the first two terms vanish. This implies that
   \[
   {\det}^{\prime}(L(G)) = \det(A(G_1,Y)) \det(A(G_2,Y)) \dfrac{d}{d\lambda}\det(\DN(G_1,Y,G_2,\lambda))\big|_{\lambda=0}
   \]
We know from Theorem \ref{Thm: pesudo-determinant of DN Map} that 
  \[
  \dfrac{d}{d\lambda}\det(\DN(G_1,Y,G_2,\lambda))\bigg|_{\lambda=0} = \dfrac{|V(G)|}{|V(Y)|} {\det}^{\prime}(\DN(G_1,Y,G_2)),
  \] completing the proof. 
\end{proof}

\subsection{Proof of Theorem \ref{theorem:schur complement formula}}
In Subsection \ref{subsection: gluing formula for pseudo-determinant}, we discussed that Theorem \ref{theorem:schur complement formula} is equivalent to the Schur complement  type formula for the number of spanning trees derived in \cite{Zhou2020}. Here, we discuss a direct proof using the argument from Subsection \ref{subsection: main proofs part-I}. Let us write
\[
C(G, \lambda) = \begin{bmatrix}
    A(\lambda) & B \\
    B^T & D(\lambda)
\end{bmatrix}
\] and let $A = A(0)$. 
Then, for $\lambda > 0$, from Proposition \ref{remark: Schur complement formula for determinant of massive Laplacian},  we have 
\begin{equation}
    \det(C(G,\lambda)) = \det(A(\lambda)) \det(\DN(G,Y,\lambda))
\end{equation}
Then, using the same argument used in Subsection \ref{subsection: main proofs part-I}, we can show that 
\begin{equation}
    {\det}'(L(G)) = \det(A) \dfrac{d}{d\lambda}\bigg|_{\lambda=0}\det(\DN(G,Y,\lambda)), 
\end{equation} and using the argument to prove Theorem  \ref{Thm: pesudo-determinant of DN Map}, we can show 
\begin{equation}
    \dfrac{d}{d\lambda}\bigg|_{\lambda=0}\det(\DN(G,Y,\lambda))  =\dfrac{|V(G)|}{|V(Y)|} {\det}'\DN(G,Y). 
\end{equation} Hence,
\[
{\det}'(L(G)) = \det(A) \dfrac{|V(G)|}{|V(Y)|} {\det}'(\DN(G,Y))
\] completing the proof of Theorem  \ref{theorem:schur complement formula}. 

\subsection{Proof of Theorem \ref{cor: main cor-2}}

Now, we prove Theorem \ref{cor: main cor-2}. We know from Theorem \ref{theorem:main theorem-1} 
\begin{equation*}
\tau(G) = \det(A_1)\det(A_2) \dfrac{{\det}^{\prime}(\DN(G_1,Y,G_2))}{|V(Y)|}
\end{equation*} and from Theorem \ref{theorem:schur complement formula}
\begin{equation*}
 \det(A_i) = |V(Y)|\dfrac{\tau(G_i)}{{\det}^{\prime}(\DN(G_i,Y))}
\end{equation*} These immediately implies
\[
\tau(G) = \tau(G_1) \tau(G_2)  C(G_1, G_2, Y). 
\]

%
%
%

\section{Examples} 
In this section, we give examples of explicit computation of Dirichlet-to-Neumann maps. We also apply gluing formula for the number of spanning trees to give closed form formula for the number of spanning trees in two different of families of graphs: generalized core satellite graphs and gluing of cycle graphs.

\subsection{Generalized Core Satellite Graphs}

Let us begin with case where two complete graphs are glued along a complete subgraph. The following lemma is very useful.   
\begin{lemma} \label{lemma:DNcomplete}
Let $l\leq n$ and $G= K_n$.   Let $Y$  be a subgraph of $G$ such that $Y = K_l$. Then,  
\[
\DN(G,Y) = \dfrac{n}{l} L({Y})
\]
\begin{proof} Let us write 
\[
L(G) = \begin{bmatrix}
    A & B \\
    B^{T} & D 
\end{bmatrix}, 
\] where $A = n \mathrm{I}_{n-l} - \boldsymbol{1}_{n-l} \boldsymbol{1}_{n-l}^{T}$, $B = -\boldsymbol{1}_{n-l} \boldsymbol{1}_{l}^{T}$,  and $D = n \mathrm{I}_{l} - \boldsymbol{1}_{l} \boldsymbol{1}_{l}^{T}$.   By Sherman-Morrison formula, 
\[
A^{-1} = \dfrac{1}{n} \mathrm{I}_{n-l} + \dfrac{1}{nl} \boldsymbol{1}_{n-l} \boldsymbol{1}_{n-l}^{T}. 
\] Hence,
\[
B^TA^{-1}B = \dfrac{n-l}{l} \boldsymbol{1}_{l} \boldsymbol{1}_{l}^{T}
\] Now,
\[
\DN(G,Y) = D - B^TA^{-1}B = n \mathrm{I}_{l}- \dfrac{n}{l}  \boldsymbol{1}_{l} \boldsymbol{1}_{l}^{T} = \dfrac{n}{l} (l  \mathrm{I}_{l}-\boldsymbol{1}_{l} \boldsymbol{1}_{l}^{T}) = \dfrac{n}{l} L(Y)
\]  as needed.  
\end{proof}
\end{lemma} 

From Proposition  \ref{prop: gluing DN maps} and Lemma \ref{lemma:DNcomplete},  we immediately have the following.  
\begin{corollary}\label{corollary:DNgluingcomplete} Let $l,m,$ and $n$ be positive integers such that $l \leq \min\{m,n\}$ and $G = K_m \cup_{K_l} K_n$.  Then 
\[
\DN(G, K_l) = \dfrac{m+n-l}{l} L(Y) = \dfrac{|V(G)|}{|V(Y)|} L(Y)
\]
\end{corollary}

Now,  we are ready to give a closed form formula for the number of spanning trees on the clique sum $K_m \cup_{K_l} K_n$.

\begin{proposition} \label{prop:clique_sum} 
The number of spanning trees on $G = K_m \cup_{K_l} K_n$ is given by
\[
\tau(G) = l m^{m-l-1}n^{n-l-1} (m+n-l)^{l-1}
\] 
\begin{proof}  From Lemma \ref{lemma:DNcomplete},  we have  
\[
{\det}^{\prime}(\DN(K_m,K_l))= m^{l-1}\quad \text{and}\quad {\det}^{\prime}(\DN(K_n,K_l))= n^{l-1}
\] From Corollary \ref{corollary:DNgluingcomplete},  we have 
\[
{\det}^{\prime}(\DN(G, K_l)) = (m+n-l)^{l-1}
\]Thus,  from Theorem  \ref{cor: main cor-2},  we get 
\begin{align*}
\tau(G) = l \tau(K_m)\tau(K_n) \dfrac{ (m+n-l)^{l-1}}{m^{l-1}n^{l-1}} = l m^{m-l-1}n^{n-l-1} (m+n-l)^{l-1}
\end{align*}
\end{proof} 
\end{proposition}

It turns out that the Lemma \ref{lemma:DNcomplete} can be generalized as follows.  
\begin{proposition} \label{proposition:dnjoin} Let $\tilde{G}$ be a graph and $Y = K_l$.  Let $G = \tilde{G} \vee Y$ be the join of $\tilde{G}$ and $Y$.   Then,
\[
\DN(G, Y)  = \dfrac{|V(G)|}{|V(Y)|} L(Y)
\]
\begin{proof} Let $m = |V(\tilde{G})|$.   Let us write
\[
L(G) = \begin{bmatrix}
            A  & B \\
            B^{T} & D
           \end{bmatrix}
\] We note that $A = L(\tilde{G}) + l \mathrm{I}_m$,   $D = L(Y) + m \mathrm{I}_l$,  $B= -\boldsymbol{1}_m \boldsymbol{1}_l^{T}$.  Since $L(\tilde{G})\boldsymbol{1}_m  =0$,  we have $A\boldsymbol{1}_m  = l \boldsymbol{1}_m $.   Thus, 
\[
\boldsymbol{1}_m^{T} A^{-1}\boldsymbol{1}_m  = l \boldsymbol{1}_m^{T}\boldsymbol{1}_m  = \dfrac{m}{l}
\] Using this we conclude: 
\begin{align*}
\DN(G,Y)  = D - B^{T}A^{-1}B &=
L(Y) + m \mathrm{I}_l - \dfrac{m}{l}  \boldsymbol{1} \boldsymbol{1}_l^{T}  \\
&= L(Y) + \dfrac{m}{l} (l  \mathrm{I}_l - \boldsymbol{1} \boldsymbol{1}_l^{T}) \\
& = L(Y)  + \dfrac{m}{l} L(Y) \\ 
& =  \dfrac{|V(G)|}{|V(Y)|} L(Y)
\end{align*}
\end{proof} 
\end{proposition}

Next,  we explore the number of spanning trees on the so-called generalized core–satellite graph, which is generalization of graph with the same name introduced in \cite{estrada2017}.  

\begin{definition}  \label{definition:coresatellitegraphs}
Let $\tilde{G}_1, \ldots,  \tilde{G}_n$  be connected graphs.  Let $G_i = \tilde{G}_i \vee K_l$, the join of $\tilde{G}_i$ and $Y$.   A generalized core–satellite graph with shared core $K_l$ and ``satellites'' $\tilde{G}_1, \ldots,  \tilde{G}_n$ is denoted by $\Theta(\tilde{G}_1, \ldots,  \tilde{G}_n,  K_l)$ and it is defined as the graph obtained by gluing $G_1, \ldots, G_n$ along $K_l$.  In other words,
\[
\Theta(\tilde{G}_1, \ldots,  \tilde{G}_n,  K_l) = G_1\cup_{K_l} \ldots \cup_{K_l}G_n 
\]

\end{definition}

\begin{remark} In \cite{estrada2017}, the authors use the join operation rather than gluing operation.  The generalized core-satellite graph defined in \cite{estrada2017} is isomorphic to the special case of Definition \ref{definition:coresatellitegraphs} if we take $\tilde{G}_i$ to be a complete graph.  More precisely,  let $\boldsymbol{s} = (s_1, \ldots, s_n)$, $\boldsymbol{\eta} = (1,  \ldots,  1)$,  and $\tilde{G}_i = K_{s_i}$ and  $\theta(l,\boldsymbol{s}, \boldsymbol{\eta})$ be the generalized core-satellite graph defined in \cite{estrada2017}.  Then,   $\Theta(\tilde{G}_1, \ldots,  \tilde{G}_n,  K_l)$ is isomorphic to $\theta(l,\boldsymbol{s}, \boldsymbol{\eta})$.   Since the graphs $\tilde{G}_i $ in Definition \ref{definition:coresatellitegraphs} are not necessarily distinct,  it covers general $\boldsymbol{\eta}$ as well.  
\end{remark} 

\noindent 
It turns out that the Dirichlet-to-Neumann map $\DN(\Theta(\tilde{G}_1, \ldots,  \tilde{G}_n,  K_l), K_l)$ can be computed explicitly.  
\begin{theorem}\label{theorem:DNcoresatellite}
 Let $G = \Theta(\tilde{G}_1, \ldots,  \tilde{G}_n,  K_l)$ and $Y = K_l$. Then,
\[
\DN(G,Y) = \dfrac{|V(G)|}{|V(Y)|} L(Y)
\]
\begin{proof} We use induction on $n$.  The case $n=1$ follows from Proposition \ref{proposition:dnjoin}.  Note that 
\[
\Theta(\tilde{G}_1, \ldots,  \tilde{G}_n, K_l) = \Theta(\tilde{G}_1, \ldots,  \tilde{G}_{n-1},K_l)\cup_{K_l} G_n  
\]
 By induction hypothesis and Proposition \ref{proposition:dnjoin}
\[
\DN( \Theta(\tilde{G}_1, \ldots,  \tilde{G}_{n-1},  K_l), K_l) = \dfrac{|V(\Theta(\tilde{G}_1, \ldots,  \tilde{G}_{n-1},  K_l))|}{|V(Y)|}L(Y) 
\] and 
\[
 \DN(G_n,Y) = \dfrac{|V(G_n)|}{|V(Y)|} L(Y)
\] 
Using Proposition \ref{prop: gluing DN maps}, we have 
\begin{align*}
\DN(G,Y)& = \DN( \Theta(\tilde{G}_1, \ldots,  \tilde{G}_{n-1},  K_l), K_l) + \DN(G_n,  K_l) - L(Y) \\
             & = \dfrac{|V(\Theta(\tilde{G}_1, \ldots,  \tilde{G}_{n-1},  K_l))|}{|V(Y)|}L(Y) + \dfrac{|V(G_n)|}{|V(Y)|} L(Y) - L(Y) \\
            &= \dfrac{|V(\Theta(\tilde{G}_1, \ldots,  \tilde{G}_{n-1},  K_l))| + |V(G_n)|- |V(Y)|}{|V(Y)|} L(Y) \\
            & =  \dfrac{|V(G)|}{|V(Y)|} L(Y)
\end{align*}
\end{proof} 
\end{theorem}

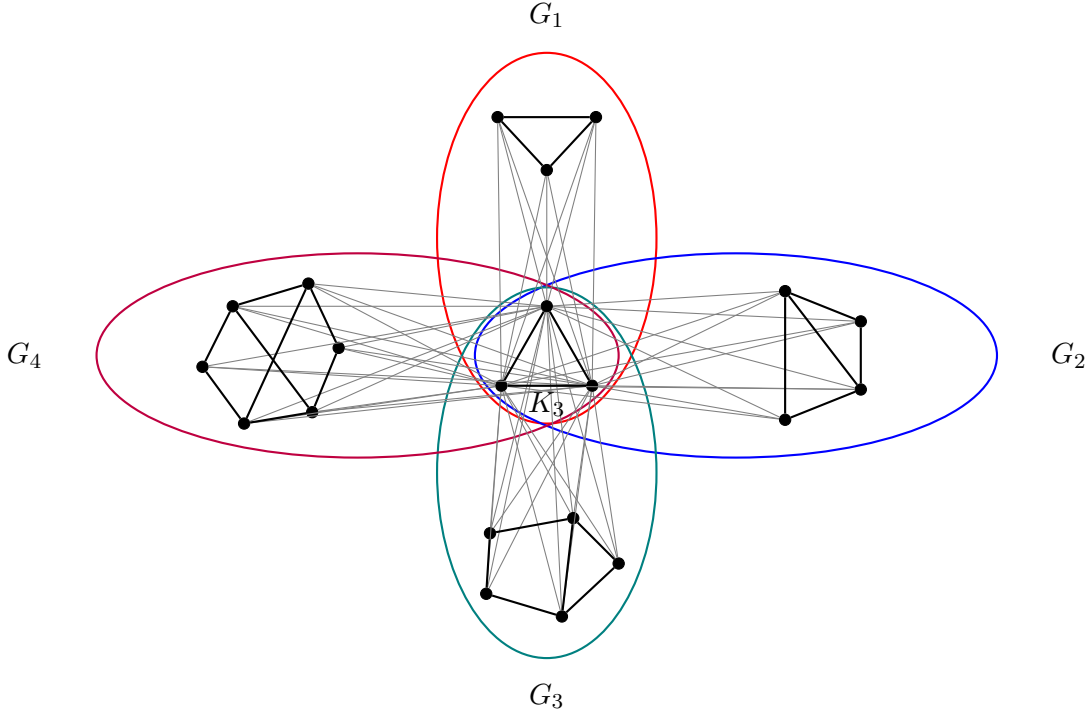
\begin{figure}[H]
    \centering
    \begin{tikzpicture}[
        vertex/.style={circle, draw, fill=black, inner sep=1.5pt},
        coreedge/.style={thick},
        satedge/.style={thick},
        joinedge/.style={gray, thin},
        region/.style={thick}
    ]

        \draw[red, region]    (0, 1.55) ellipse (1.45 and 2.45);
        \draw[blue, region]   (2.5, 0) ellipse (3.45 and 1.35);
        \draw[teal, region]   (0,-1.55) ellipse (1.45 and 2.45);
        \draw[purple, region] (-2.5, 0) ellipse (3.45 and 1.35);

        \node[vertex] (y1) at (0, 0.65) {};
        \node[vertex] (y2) at (-0.6, -0.4) {};
        \node[vertex] (y3) at (0.6, -0.4) {};

        \node[vertex] (a1) at (-0.65, 3.15) {};
        \node[vertex] (a2) at (0.65, 3.15) {};
        \node[vertex] (a3) at (0, 2.45) {};

        \node[vertex] (b1) at (3.15, 0.85) {};
        \node[vertex] (b2) at (4.15, 0.45) {};
        \node[vertex] (b3) at (4.15, -0.45) {};
        \node[vertex] (b4) at (3.15, -0.85) {};

        \node[vertex] (c1) at (-0.8, -3.15) {};
        \node[vertex] (c2) at (0.2, -3.45) {};
        \node[vertex] (c3) at (0.95, -2.75) {};
        \node[vertex] (c4) at (0.35, -2.15) {};
        \node[vertex] (c5) at (-0.75, -2.35) {};

        \node[vertex] (d1) at (-3.15, 0.95) {};
        \node[vertex] (d2) at (-4.15, 0.65) {};
        \node[vertex] (d3) at (-4.55, -0.15) {};
        \node[vertex] (d4) at (-4.0, -0.9) {};
        \node[vertex] (d5) at (-3.1, -0.75) {};
        \node[vertex] (d6) at (-2.75, 0.1) {};

        \foreach \s in {a1,a2,a3,b1,b2,b3,b4,c1,c2,c3,c4,c5,d1,d2,d3,d4,d5,d6} {
            \foreach \y in {y1,y2,y3} {
                \draw[joinedge] (\s) -- (\y);
            }
        }

        \draw[satedge] (a1) -- (a2);
        \draw[satedge] (a2) -- (a3);
        \draw[satedge] (a3) -- (a1);

        \draw[satedge] (b1) -- (b2);
        \draw[satedge] (b2) -- (b3);
        \draw[satedge] (b3) -- (b4);
        \draw[satedge] (b4) -- (b1);
        \draw[satedge] (b1) -- (b3);

        \draw[satedge] (c1) -- (c2);
        \draw[satedge] (c2) -- (c3);
        \draw[satedge] (c3) -- (c4);
        \draw[satedge] (c4) -- (c5);
        \draw[satedge] (c5) -- (c1);
        \draw[satedge] (c2) -- (c4);

        \draw[satedge] (d1) -- (d2);
        \draw[satedge] (d2) -- (d3);
        \draw[satedge] (d3) -- (d4);
        \draw[satedge] (d4) -- (d5);
        \draw[satedge] (d5) -- (d6);
        \draw[satedge] (d6) -- (d1);
        \draw[satedge] (d1) -- (d4);
        \draw[satedge] (d2) -- (d5);

        \node at (0, 4.5) {$G_1$};
        \node at (6.9, 0) {$G_2$};
        \node at (0, -4.5) {$G_3$};
        \node at (-6.9, 0) {$G_4$};

        \node at (0, -0.64) {$K_3$};
        \draw[coreedge] (y1) -- (y2);
        \draw[coreedge] (y2) -- (y3);
        \draw[coreedge] (y1) -- (y3);

    \end{tikzpicture}
    \caption{A generalized core--satellite graph $\Theta(\tilde{G}_1, \tilde{G}_2, \tilde{G}_3,  \tilde{G}_4,  K_3)$ with shared core $K_3$.}
    \label{Fig:CoreSatelliteGraphK3}
\end{figure}

In the next theorem,  we give a closed form expression for the number of spanning trees of $\Theta(\tilde{G}_1, \ldots,  \tilde{G}_n,  K_l)$ in terms of the number of spanning trees on $G_1,\ldots, G_n$. 

\begin{theorem} \label{theorem:maincorestatellite}

The number of spanning trees of  $\Theta(\tilde{G}_1, \ldots,  \tilde{G}_n,  K_l)$ is given by 
\begin{equation}
\tau(\Theta(\tilde{G}_1, \ldots,  \tilde{G}_n,  K_l)) = l^{n-1} \left(\dfrac{|V(\Theta(\tilde{G}_1, \ldots,  \tilde{G}_n,  K_l))|}{\prod_{i=1}^{n}|V(G_i)|}\right)^{l-1}\prod_{i=1}^{n}\tau(G_i)
\end{equation} 
\begin{proof} We use induction on $n$.  For $n=1$,  the result is trivial.  Using
\[
\Theta(\tilde{G}_1, \ldots,  \tilde{G}_n, K_l) = \Theta(\tilde{G}_1, \ldots,  \tilde{G}_{n-1},K_l)\cup_{K_l} G_n  
\] and Theorem  \ref{cor: main cor-2},  we get 
\begin{align}\label{eqon:aux}
\begin{array}{lll}
\tau(\Theta(\tilde{G}_1, \ldots,  \tilde{G}_n, K_l))& = l \tau(\Theta(\tilde{G}_1, \ldots,  \tilde{G}_{n-1}, K_l))\tau(G_n) \\
& \dfrac{{\det}^{\prime}(\DN(\Theta(\tilde{G}_1, \ldots,  \tilde{G}_n, K_l), K_l))}{{\det}^{\prime}(\DN(\Theta(\tilde{G}_1, \ldots,  \tilde{G}_{n-1}, K_l), K_l)){\det}^{\prime}(\DN(G_n,K_l))}
\end{array}
\end{align}
From Theorem \ref{theorem:DNcoresatellite},  it follows that 
\begin{align}\label{eqon:aux1}
\begin{array}{lll}
&{\det}^{\prime}(\DN(\Theta(\tilde{G}_1, \ldots,  \tilde{G}_n, K_l), K_l)) = |V(\Theta(\tilde{G}_1, \ldots,  \tilde{G}_n, K_l))|^{l-1} \\\\
& {\det}^{\prime}(\DN(G_n,  K_l)) = |V(G_n)|^{l-1}
\end{array}
\end{align} 
Using induction hypothesis and Theorem \ref{theorem:DNcoresatellite}, we have 
\begin{equation}\label{eqon:aux2}
\dfrac{\tau(\Theta(\tilde{G}_1, \ldots,  \tilde{G}_{n-1}, K_l))}{{\det}^{\prime}(\DN(\Theta(\tilde{G}_1, \ldots,  \tilde{G}_{n-1}, K_l), K_l))} =  l^{n-2} \left(\dfrac{1}{\prod_{i=1}^{n-1}|V(G_i)|}\right)^{l-1}\prod_{i=1}^{n-1}\tau(G_i)
\end{equation} 
The proof of the theorem is completed by using relations (\ref{eqon:aux1}) and (\ref{eqon:aux2}) in (\ref{eqon:aux}). 
\end{proof}
\end{theorem}

As a consequence,  we give a new proof of the following recent result \cite{estrada2017, yang2026}.   
\begin{corollary}\label{corollary:Core-Satellite} Let $\tilde{G}_i = K_{s_i}$,  $i=1, \ldots, n$.   
\[
\tau(\Theta(\tilde{G}_1, \ldots,  \tilde{G}_n, K_l))=  l^{n-1}|V(\Theta(\tilde{G}_1, \ldots,  \tilde{G}_n, K_l))|^{l-1} \prod_{i=1}^{n}(l+s_i)^{(s_i-1)}
\] 
\begin{proof} Let $n_i = l + s_i$.  Then $G_i$ is isomorphic to $K_{n_i}$.  Hence,  By Caley's formula, \\
$\tau (G_i) = n_i^{n_i-2}$.   In this case,
\[
\dfrac{\prod_{i=1}^{n} \tau(G_i)} {\prod_{i=1}^{n}|V(G_i)|^{l-1}} =\prod_{i=1}^{n} \dfrac{n_i^{n_i-2}}{n_i^{l-1}} = \prod_{i=1}^{n} n_i^{s_i-1}=\prod_{i=1}^{n} (l+s_i)^{s_i-1}
\] Now,  the proof follows immediately from Theorem \ref{theorem:maincorestatellite}. 
\end{proof} 
\end{corollary}

\begin{note} The statement in Corollary \ref{corollary:Core-Satellite} appears in \cite{yang2026} which immediately follows from the explicit computation of spectrum of the Laplacian of $G_k$ in \cite{estrada2017} and the well known Matrix-Tree theorem.   
\end{note}

\begin{remark} While the complete spectrum of Laplacian of a core-satellite graph is computed in \cite{estrada2017}.  Our proof of Corollary \ref{corollary:Core-Satellite}  suggests that the spectrum can be computed inductively.  
\end{remark}

\begin{remark} It is possible to give a proof of Theorem  \ref{theorem:DNcoresatellite} using properties of the number of spanning trees under the join operations.  Here, our focus was to utilize the gluing formula for spanning trees. 
\end{remark}

\subsection{Gluing Cycle Graphs} Here,  we give a closed form formula for the number of spanning trees when a graph is formed by gluing two cycle graphs. Let $C_n$ denote cycle graph with n vertices.  Let the vertices be  $v_1,\ldots, v_n$.  Let $P_l$ be a path graph inside $C_n$ with $l$ vertices. 
\begin{lemma} \label{lemma:DNpath} Let
\[
L(C_n)  = \begin{bmatrix} A & B\\
                                     B^T & D
                 \end{bmatrix} 
\] Then, 
\begin{itemize}
    \item[(a)] The Dirichlet-to-Neumann map is given by
    \begin{equation}
   \DN(C_n, P_l) = D - \tilde{D} 
\end{equation}
where $\tilde{D}$ is $l\times l$ matrix such that $\tilde{D}_{11} = (n-l)/(n-l+1), \tilde{D}_{1l} = \tilde{D}_{l1} = 1/(n-l+1)$, and all other components $0$. 
\item[(b)] The pseudo-determinant of $\DN(C_n, P_l)$ is given by 
\[
{\det}^{\prime}(\DN(C_n, P_l)) = \dfrac{nl}{n-l+1}
\]
\end{itemize}
\begin{proof} For (a),  we note that $A$ is $(n-l) \times (n-l)$ tri-diagonal matrix with main diagonal consisting of all 2's; and sub and super diagonal consisting of all -1's. Hence, from \cite{da2001}, it follows that $A^{-1}_{ij}$ is given by 
\[
 A^{-1}_{ij} = \begin{cases} \dfrac{i((n-l)-j+1)}{n-l+1} \quad \text{if} \quad i \leq j \\
             \dfrac{j((n-l)-i+1)}{n-l+1} \quad \text{if} \quad i \geq  j
\end{cases}
\] Note that $B$ is $(n-l)\times l$ matrix such $B_{1l} = -1$ and $B_{n-l1} =-1$ and all other components $0$. Hence, 
\[
\tilde{D} = B^TA^{-1}B
=\begin{bmatrix}
        \frac{n-l}{n-l+1} & 0 & \cdots  & 0 & \frac{1}{n-l+1}\\
        0 & \ddots & && 0\\
        \vdots & &\ddots  & & \vdots\\
        0 & & & \ddots & 0\\
        \frac{1}{n-l+1} & 0 & \cdots & 0 & \frac{n-l}{n-l+1}
    \end{bmatrix}
\]
 For (b), by Kirchhoff matrix tree theorem \cite{Klee2019}: 
 \[
 {\det}^{\prime}(\DN(C_n,P_l))  = l M_{ll}
 \] where $M_{ll}$ is $(l,l)$ minor of $\DN(C_n,P_l)$. Note that $M_{ll}$ is determinant of a tri-diagonal matrix and in this case it is given  $M_{ll} = \dfrac{n}{n-l+1}$
 as needed. 
\end{proof}
\end{lemma}

\begin{remark}\label{remark:dettridiagonal} Let $M$ be a $k \times k$ tri-diagonal matrix whose main diagonal is given by $(\alpha,2, \ldots, 2)$ and sub and super diagonal are give by $(-1, \ldots, -1)$, then
\[
\det(M) = k\alpha - (k-1)
\]
    
\end{remark}

 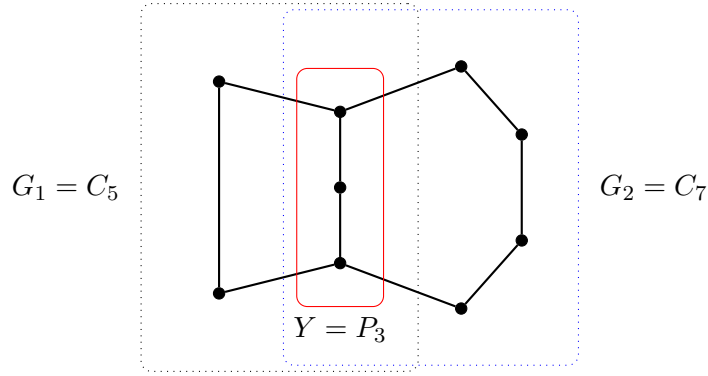
\begin{figure}[H]
    \centering
    \usetikzlibrary{fit,positioning}

    \begin{tikzpicture}[
        vertex/.style={circle, draw, fill=black, inner sep=1.5pt},
        edge/.style={thick}
    ]

        \node[vertex] (y1) at (0, 1) {};
        \node[vertex] (y2) at (0, 0) {};
        \node[vertex] (y3) at (0,-1) {};

        \node[vertex] (a1) at (-1.6, 1.4) {};
        \node[vertex] (a2) at (-1.6,-1.4) {};

        \node[vertex] (b1) at (1.6, 1.6) {};
        \node[vertex] (b2) at (2.4, 0.7) {};
        \node[vertex] (b3) at (2.4,-0.7) {};
        \node[vertex] (b4) at (1.6,-1.6) {};

        \draw[edge] (y1) -- (y2);
        \draw[edge] (y2) -- (y3);

        \draw[edge] (y1) -- (a1);
        \draw[edge] (a1) -- (a2);
        \draw[edge] (a2) -- (y3);

        \draw[edge] (y1) -- (b1);
        \draw[edge] (b1) -- (b2);
        \draw[edge] (b2) -- (b3);
        \draw[edge] (b3) -- (b4);
        \draw[edge] (b4) -- (y3);

        \node[draw=red, rounded corners, fit=(y1)(y2)(y3), inner sep=14pt] (boxY) {};
        \node[draw=black, dotted, rounded corners, fit=(y1)(y2)(y3)(a1)(a2), inner sep=27pt] (boxGOne) {};
        \node[draw=blue, dotted, rounded corners, fit=(y1)(y2)(y3)(b1)(b2)(b3)(b4), inner sep=19pt] (boxGTwo) {};

        \node[below=0pt of boxY] {$Y=P_3$};
        \node[left=4pt of boxGOne] {$G_1=C_5$};
        \node[right=4pt of boxGTwo] {$G_2=C_7$};

    \end{tikzpicture}
    \caption{$C_5$ and $C_7$ glued along a path graph $P_3$.}
    \label{Fig:GluingCyclesAlongP3}
\end{figure}

\begin{proposition} Let $l < \min\{m,n\}$ and $G = C_m \cup_{P_l} C_n$.  Then,
\begin{itemize}
    \item[(a)] The Dirichlet-to-Neumann map $\DN(G, P_l)$ is given by
    \[
    \DN(G, P_l) = L(Y) + \hat{D}
    \] where $\hat{D}$ is $l\times l$ matrix such that 
\[\hat{D}_{11} = \hat{D}_{ll}= \frac{1}{(n-l+1)} + \frac{1}{(m-l+1)}; \,\,\hat{D}_{1l} = \hat{D}_{l1} = - \hat{D}_{11}\]
and all other components $0$: 
  \[\hat{D} 
=\left(\dfrac{1}{m-l+1} + \dfrac{1}{n-l+1}\right)
\begin{bmatrix}
        1 & 0 & \cdots  & 0 & -1\\
        0 & \ddots & && 0\\
        \vdots & &\ddots  & & \vdots\\
        0 & & & \ddots & 0\\
       -1 & 0 & \cdots & 0 & 1
    \end{bmatrix}
\] In particular,
\[
{\det}^{\prime}(\DN(G,P_l)) = \frac{l(mn- (l-1)^2)}{(m-l+1)(n-l+1)}
\]
    \item[(b)]The number of spanning trees on $G$ is given by
    \[
    \tau(G) = mn - (l-1)^2
    \]
\end{itemize}
    \begin{proof} The first part of (a) directly follows from Proposition \ref{prop: gluing DN maps} and Lemma \ref{lemma:DNpath}. For the second part of (a) we proceed as in Lemma \ref{lemma:DNpath} and use $(l,l)$ minor. Now using Remark \ref{remark:dettridiagonal} to compute this minor leads to the desired result.  

        For part (b), we use Gluing relation for spanning trees Theorem  \ref{cor: main cor-2}. 
  \[
  \tau(G) = l \tau(C_n)\tau(C_m) \dfrac{\frac{l(mn- (l-1)^2))}{(m-l+1)(n-l+1)}}{\frac{ml}{(m-l+l)} \frac{nl}{(n-l+1)}} = mn- (l-1)^2
  \]      
    \end{proof}
\end{proposition}

\bibliographystyle{alpha}
\bibliography{graph-theory.bib}

\vskip10mm
S\MakeLowercase{HIVJYOT} B\MakeLowercase{RAR}\\
D\MakeLowercase{EPARTMENT} Of C\MakeLowercase{OMPUTER} S\MakeLowercase{CIENCE}\\
C\MakeLowercase{ALIFORNIA} S\MakeLowercase{TATE} U\MakeLowercase{NIVERSITY}, S\MakeLowercase{ACRAMENTO}\\
\textit{E-mail address: } \texttt{sbrar@csus.edu}
\newline

Sheng-Chang Chen\\
Department of Mathematics and Statistics\\
C\MakeLowercase{ALIFORNIA} S\MakeLowercase{TATE} U\MakeLowercase{NIVERSITY}, S\MakeLowercase{ACRAMENTO}\\
\textit{E-mail address: }\texttt{sheng-changchen@csus.edu}
\newline

Sayonita Ghosh Hajra\\
Department of Mathematics and Statistics\\
C\MakeLowercase{ALIFORNIA} S\MakeLowercase{TATE} U\MakeLowercase{NIVERSITY}, S\MakeLowercase{ACRAMENTO}\\
\textit{E-mail address: }\texttt{sayonita.ghoshhajra@csus.edu}
\newline

Santosh Kandel\\
Department of Mathematics and Statistics\\
C\MakeLowercase{ALIFORNIA} S\MakeLowercase{TATE} U\MakeLowercase{NIVERSITY}, S\MakeLowercase{ACRAMENTO}\\
\textit{E-mail address: }\texttt{kandel@csus.edu}

\end{document}